\documentclass[11pt,a4paper,draft]{article}

\usepackage[ngerman,english]{babel}
\usepackage[utf8]{inputenc}

\usepackage[DIV=9]{typearea}

\usepackage[affil-it]{authblk}

\usepackage{xparse}
\usepackage{amsmath, amsfonts, amssymb, amsthm}
\usepackage{mathtools}
\usepackage{enumitem}

\usepackage[noadjust]{cite}

\usepackage[capitalise]{cleveref}

\crefname{equation}{}{}

\newcommand{\R}{\mathbb{R}}

\newcommand{\N}{\mathbb{N}}

\newcommand{\ee}{\mathrm{e}}

\DeclareDocumentCommand\dd{ o g d() }{
	\IfNoValueTF{#2}{
		\IfNoValueTF{#3}
			{\mathrm{d}\IfNoValueTF{#1}{}{^{#1}}}
			{\mathinner{\mathrm{d}\IfNoValueTF{#1}{}{^{#1}}\argopen(#3\argclose)}}
		}
		{\mathinner{\mathrm{d}\IfNoValueTF{#1}{}{^{#1}}#2} \IfNoValueTF{#3}{}{(#3)}}
	}

\newcommand{\dx}{\dd{x}}
\newcommand{\dy}{\dd{y}}
\newcommand{\dz}{\dd{z}}

\newcommand{\del}{\partial}
\newcommand{\eps}{\varepsilon}

\newcommand{\vcc}{\vcentcolon}

\DeclarePairedDelimiter\abs{\lvert}{\rvert}
\DeclarePairedDelimiter\norm{\Vert}{\rVert}

\theoremstyle{plain}
\newtheorem{theorem}{Theorem}[section]
\newtheorem{lemma}[theorem]{Lemma}
\newtheorem{proposition}[theorem]{Proposition}

\theoremstyle{definition}
\newtheorem{definition}[theorem]{Definition}

\theoremstyle{remark}
\newtheorem{remark}[theorem]{Remark}

\numberwithin{equation}{section}

\title{Precise tail behaviour of self-similar profiles with infinite mass for Smoluchowski's coagulation equation}
\author{Sebastian Throm
  \thanks{E-mail: \texttt{throm@ma.tum.de}}}
\affil{Faculty of Mathematics, Technical University of Munich, Boltzmannstraße 3, 85748 Garching bei München, Germany}
\date{}

\begin{document}
 
\maketitle

\begin{abstract}
 We consider self-similar profiles to Smoluchowski's coagulation equation for which we derive the precise asymptotic behaviour at infinity. More precisely, we look at so-called fat-tailed profiles which decay algebraically and as a consequence have infinite total mass. The results only require mild assumptions on the coagulation kernel and thus cover a large class of rate kernels.
\end{abstract}

\section{Introduction}

Smoluchowski's coagulation equation is a kinetic model which describes the irreversible aggregation of clusters in particle systems that may originate from many different applications in physics, chemistry or biology. This model was originally derived by Marian von Smoluchowski in 1916 in the context of coagulation of gold particles in a colloidal solution which move according to Brownian motion (\cite{Smo16,Smo17}). The basic assumptions for the derivation of the equation are that each particle is completely characterised by a scalar quantity $x\in(0,\infty)$ which is usually denoted as the size or mass of the cluster. Moreover, one assumes that the state of the whole system is fully described by the density $\phi(x,t)$ of clusters of size $x$ at time $t$. Finally, one assumes that the system is dilute in the sense that in each coagulation process only two clusters are involved which means that the effect of higher order collisions can be neglected as a rare event.

The time evolution of such systems is described by Smoluchowski's coagulation equation which reads as
\begin{equation}\label{eq:Smol}
 \del_{t}\phi(x,t)=\frac{1}{2}\int_{0}^{x}K(x-y,x)\phi(x-y,t)\phi(y,t)\dy-\phi(x,t)\int_{0}^{\infty}K(x,y)\phi(y,t)\dy.
\end{equation}
The first integral on the right-hand side is denoted as the gain term and accounts for the creation of clusters of size $x$ from clusters of sizes $x-y$ and $y$ while the factor $1/2$ is due the symmetry of the process. The second integral on the right-hand side however corresponds to the loss of particles of size $x$ due to the aggregation process and may thus be denoted as the loss term. 

Besides the specific application which Smoluchowski had in mind when he derived~\eqref{eq:Smol}, the same model is used to describe a broad range of other phenomena in many different fields. More details on the model and its applications can be found in~\cite{LaM04,Dra72}.

\subsection{Assumptions on the kernel $K$}

As one can see from~\eqref{eq:Smol}, the dynamics of the system is completely determined by the integral kernel $K$. One prominent example of such a kernel is the one that Smoluchowski derived himself for the gold particles and which reads as
\begin{equation}\label{eq:Smol:kernel}
 K(x,y)=\bigl(x^{1/3}+y^{1/3}\bigr)\bigl(x^{-1/3}+y^{-1/3}\bigr).
\end{equation}
The quantity $x^{1/3}$ corresponds to the radius, while $x^{-1/3}$ relates to the diffusion constant of a spherical cluster of size $x$.

Throughout this work we assume that the integral kernel $K$ is a continuous function on $\R_{+}\times\R_{+}$ which is symmetric and homogeneous of degree $\lambda\in(-1,1)$, i.e.\@
\begin{equation}\label{eq:Ass:ker:1}
 K(x,y)=K(y,x) \quad \text{and}\quad K(rx,ry)=r^{\lambda}K(x,y) \quad \text{for all }r,x,y\in\R_{+}.
\end{equation}
Moreover, we require some upper and lower bound on the kernel $K$ and we precisely assume that there exist constants $c_{*}, C_{*}>0$ and $\alpha,\beta\in(-1,1)$ with $\alpha\leq \beta$ and $\lambda=\alpha+\beta$ as well as $b,B\in (0,\infty)$ with $b<B$ such that it holds
\begin{equation}\label{eq:Ass:ker:2}
 \inf_{x,y\in[b,B]}K(x,y)\geq c_{*}\qquad \text{and}\qquad 0\leq K(x,y)\leq C_{*}\bigl(x^{\alpha}y^{\beta}+x^{\beta}y^{\alpha}\bigr).
\end{equation}
\begin{remark}
Note that for $c_{*}=2$, $C_{*}=3$ and $\alpha=-1/3$, $\beta=1/3$ this covers in particular the case~\eqref{eq:Smol:kernel} for each choice of $b,B\in(0,\infty)$ with $b<B$.
\end{remark}
\begin{remark}
 Note that the well-posedness of~\eqref{eq:Smol} has been shown for example in~\cite{Nor99,FoL06a} for a broad range of rate kernels.
\end{remark}

\subsection{Self-similar profiles}\label{Sec:profiles}

One particular point of interest concerns the behaviour of solutions to~\eqref{eq:Smol} for $t\to\infty$. For integral kernels satisfying~\eqref{eq:Ass:ker:1} one expects that the long-time behaviour of solutions to~\eqref{eq:Smol} is self-similar. Precisely this means that there exist solutions $\phi$ for~\eqref{eq:Smol} which are of the form
\begin{equation}\label{eq:ansatz}
 \phi(x,t)=\frac{1}{(s(t))^{\theta}}f\biggl(\frac{x}{s(t)}\biggr)\quad \text{with }s(t)\longrightarrow \infty \text{ for }t\longrightarrow \infty.
\end{equation}
The so-called scaling hypothesis then states that the long-time behaviour of solutions $\phi$ to~\eqref{eq:Smol} is given by the profile $f$ in the sense that
\begin{equation}\label{eq:longtime:convergence}
 \bigl(s(t)\bigr)^{\theta}\phi(s(t)x,t)\longrightarrow f(x)\quad \text{for }t\longrightarrow \infty.
\end{equation}
This conjecture is unproven for kernels that one typically finds in applications such as~\eqref{eq:Smol:kernel}. However, there are three specific kernels namely the so-called solvable ones (i.e.\@ $K\equiv 2$, $K(x,y)=x+y$ and $K(x,y)=xy$) for which the scaling hypothesis is known to be true due to the fact that explicit solution formulas are available in terms of the Laplace transform. Thus, these kernels allow to verify~\eqref{eq:longtime:convergence} in terms of weak convergence of measures (see~\cite{MeP04}). Moreover, in~\cite{MeP04} it is also shown that there exists a unique profile $f_{1}$ with exponential decay at infinity while on the other hand one has a whole family of self-similar profiles $f_{\rho}$ which in general decay algebraically. If one considers for example the constant kernel $K\equiv 2$, one finds that up to rescaling $f_{1}$ is given by $f_{1}(x)=\ee^{-x}$, while the profiles $f_{\rho}$ for $\rho\in(0,1)$ have the asymptotic behaviour $f_{\rho}\sim x^{-1-\rho}$ for $x\to\infty$.

The existence of self-similar profiles with fast decay has been established for many non-solvable kernels in~\cite{FoL05,EMR05} while existence of fat-tailed profiles could also be established recently for a broad class of coagulation kernels in~\cite{NiV13a,NTV16a}. The general approach for the latter has been to apply some fixed-point argument to the set of non-negative measures $\mu$ such that $\int_{0}^{R}x\mu(\dx)\sim R^{1-\rho}$ for $R\to \infty$. This means that the fat-tailed profiles have already been constructed as measures with a specific (weak) tail behaviour at infinity. Since there is up to now no general uniqueness statement for self-similar profiles available, one might think of the possibility that there exist other profiles with a different decay at infinity.

That this cannot be the case is exactly the goal of this work, i.e.\@ we show rigorously that each fat-tailed self-similar profile (up to a certain rescaling) has the decay behaviour as described before. 

As already indicated, except for special cases the question of uniqueness of self-similar profiles is completely open for non-solvable kernels. However, one example where uniqueness of self-similar profiles could be established recently is a perturbed model of the constant kernel, i.e.\@ $K=2+\eps W(x,y)$ with sufficiently small $\eps>0$ and $W$ being analytic and behaving in a certain sense like $(x/y)^{\alpha}+(y/x)^{\alpha}$ with $\alpha\in[0,1/2)$ (for the precise conditions we refer to~\cite{NTV16}). For such kernels it is possible to show uniqueness of self-similar profiles both in the case of finite mass (\cite{NTV16}) and for fat-tailed profiles (\cite{Thr16}). The proof is based in both cases on Laplace transform methods and it turns out to be convenient to normalise the profiles according to their tail behaviour which might also be seen as a motivation for this work.

\subsection{The equation for self-similar profiles}\label{Sec:equation:profiles}

If one plugs the ansatz~\eqref{eq:ansatz} into~\eqref{eq:Smol} one obtains by typical scaling arguments that, in order to obtain $s(t)\to \infty$ for $t\to \infty$, it should hold $\theta> 1+\lambda$ while $f$ solves the equation
\begin{equation}\label{eq:formal:self:sim}
 \theta f(x)+xf'(x)+\frac{1}{2}\int_{0}^{x}K(x-y,y)f(x-y)f(y)\dy-f(x)\int_{0}^{\infty}K(x,y)f(y)\dy=0.
\end{equation}
For the scaling function $s(t)$ one can assume without loss of generality that it holds
\begin{equation*}
 s(t)=\bigl((\theta-1-\lambda)t\bigr)^{1/(\theta-1-\lambda)}.
\end{equation*}
It is well-known that for $K(x,y)=x^{\alpha}y^{\beta}+x^{\beta}y^{\alpha}$ with $\alpha,\beta>0$ self-similar profiles have a singular behaviour close to zero, i.e.\@ $f(x)\sim x^{-1-\lambda}$ for $x\to 0$ (\cite{DoE88}). As a consequence the integrals in~\eqref{eq:formal:self:sim} are not well-defined and one has to exploit some cancellation between them. Thus, in order to obtain a regularised equation for self-similar profiles, we multiply~\eqref{eq:formal:self:sim} by $x$ and integrate to get after some elementary rearrangement that
\begin{equation}\label{eq:self:sim}
 x^{2}f(x)=(1-\rho)\int_{0}^{x}yf(y)\dy+\int_{0}^{x}\int_{x-y}^{\infty}yK(y,z)f(y)f(z)\dz\dy.
\end{equation}
Here we also introduced the parameter $\rho=\theta-1$ while due to our restrictions on $\theta$ we have $\rho>\lambda$. One fundamental property of~\eqref{eq:Smol} is that the total mass, i.e.\@ the first moment $\int_{(0,\infty)}x\phi(x,t)\dx$ is at least formally conserved over time provided that this quantity is finite for $t=0$. In this case one also expects the scaling ansatz~\eqref{eq:ansatz} to have constant mass which also fixes the parameter $\theta=2$ or equivalently $\rho=1$ and the first moment of $f$ is finite. Thus, the natural range for the parameter $\rho$ is $(\lambda,1]$. Since we will restrict in this work exclusively on the case of profiles with infinite mass we will only consider $\rho\in(\lambda,1)$. 

Note that in the case of finite mass, i.e.\@ $\rho=1$ the first term on the right-hand side of~\eqref{eq:self:sim} vanishes and the tail behaviour of the profile is completely determined by the non-linear integral operator on the right-hand side. Formal considerations in~\cite{DoE88} suggest that up to rescaling each self-similar profile $f$ with finite mass behaves as $f(x)=Ax^{-\lambda}\ee^{-x}$ for $x\to\infty$, where the constant $A$ depends on the rate kernel $K$. Except for a small class of kernels with homogeneity zero (see~\cite{NiV14}) this precise behaviour could not yet been established for non-solvable kernels but~\cite{EsM06,FoL06} provide exponential upper and lower bounds which are in accordance with the conjectured asymptotic behaviour. 

On the other hand, for profiles with infinite mass, i.e.\@ for $\rho\in(\lambda,1)$ the asymptotic behaviour for large cluster sizes is completely different. Precisely, for $\rho\in(\lambda,1)$ the second term on the right-hand side of~\eqref{eq:self:sim} is of lower order for $x\to\infty$. Thus, if we neglect this term, the large-mass behaviour of $f$ to first order is formally given by a simple linear ordinary differential equation, i.e.\@ one obtains $f(x)\sim (1-\rho)x^{-1-\rho}$ for $x\to \infty$.

The goal of this work is to prove that for kernels $K$ satisfying~\cref{eq:Ass:ker:1,eq:Ass:ker:2} each self-similar profile necessarily has this behaviour. Due to the different properties of~\eqref{eq:self:sim} for fat-tailed profiles our proof covers a much broader class of kernels compared to the case of finite mass considered in~\cite{NiV14}.

We also emphasise at this point that in the existence results~\cite{NiV13a,NTV16a} mentioned above, the asymptotic behaviour of the constructed profiles is also determined but the important difference to the present result is that there, this behaviour has already been encoded in weak form in the construction process (see also Section~\ref{Sec:profiles}) whereas here, we start with a general notion of self-similar profiles. This point is in particular important due to the lack of a uniqueness statement for most coagulation kernels.

Finally, we note that the asymptotic behaviour close to zero in the finite-mass-case could be established rigorously in~\cite{FoL06,CaM11} for specific power law kernels. 

\subsection{Main results}

Before we give the precise definition of self-similar profiles that we will use here, we first note that from the formal considerations in the previous section we expect that the scaling solutions $f$ satisfy $f(x)\sim x^{-1-\rho}$ as $x\to\infty$. Thus, due to the upper bound on $K$ in~\eqref{eq:Ass:ker:2}, in order to obtain a finite integral in the second term on the right-hand side of~\eqref{eq:self:sim}, we should at least require that $\rho>\beta$. Together with the condition $\rho>\lambda$ that one deduces by scaling arguments (see Section~\ref{Sec:equation:profiles}) one is thus led to the following notion of self-similar profiles.

\begin{definition}\label{Def:profile}
 For an integral kernel $K$ satisfying assumptions~\cref{eq:Ass:ker:1,eq:Ass:ker:2} and $\rho\in(\max\{\lambda,\beta\},1)$ we denote a function $f\in L^{1}_{\mathrm{loc}}\bigl((0,\infty)\bigr)$ a self-similar profile (to~\eqref{eq:Smol}) or equivalently a solution to~\eqref{eq:self:sim} if $f$ is non-negative, $f\not\equiv 0$ and $f$ satisfies~\eqref{eq:self:sim} almost everywhere. Additionally, we require that $f$ satisfies $xf(x)\in L^{1}_{\mathrm{loc}}\bigl([0,\infty)\bigr)$ and there exists $\gamma\geq \beta$ such that $\gamma>\lambda$ and it holds
 \begin{equation}\label{eq:min:moment:bd}
  \int_{1}^{\infty}x^{\gamma}f(x)\dx<C(\gamma,f).
 \end{equation}
\end{definition}

\begin{remark}\label{Rem:scale:invariance}
 Note that if $f$ is a self-similar profile according to Definition~\ref{Def:profile} then also the rescaled function $\tilde{f}(x)=a^{1+\lambda}f(ax)$ for all $a>0$ is a self-similar profile.
\end{remark}

\begin{remark}
 Note that we consider here rather general kernels with homogeneity $\lambda\in(-1,1)$. Thus, if $\beta<0$ our results are even true for certain negative values of $\rho$, i.e.\@ profiles $f$ which decay slower than $1/x$ at infinity.
\end{remark}

\begin{remark}
 Note that the results presented here are in principle true and can be shown in the same way if the self-similar profiles are only assumed to be non-negative Radon measures instead of $L^{1}$-functions. However, in order to simplify the presentation and to avoid certain technicalities we do not focus on this.
\end{remark}

Before we state the main results of this work let us give some comment on the integral bound~\eqref{eq:min:moment:bd}. For this, we first recall that due to the assumptions on the kernel $K$ we can write the homogeneity $\lambda=\alpha+\beta$ with $-1<\alpha\leq \beta<1$. From the upper bound on $K$ in~\eqref{eq:Ass:ker:2} it follows that, in order to obtain a convergent integral in~\eqref{eq:self:sim}, we necessarily need $\int_{1}^{\infty}x^{\beta}f(x)\dx<\infty$. Thus, if $\alpha<0$ we have $\lambda=\alpha+\beta<\beta$ and the requirement~\eqref{eq:min:moment:bd} arises naturally with $\gamma=\beta$. 

On the other hand, if $\alpha>0$ one can easily check that the explicit power law $\widehat{f}(x)=A_{\rho,\lambda}x^{-1-\lambda}$ satisfies~\eqref{eq:self:sim} for an appropriate choice of the constant $A_{\rho,\lambda}>0$, while this function obviously has not the expected decay behaviour at infinity. Thus, in order to rule out this special solution, we assume a higher integrability condition at infinity, while the minimal choice would in principle be $\gamma=\lambda$. However, it turns out that in order to get the iterative procedure that we use here started, we need slightly more integrability at infinity which is the reason why we require $\gamma>\lambda$.

Although in the case $\alpha=0$ we do not have the occurrence of an explicit power-law solution, this is a borderline case and we also have to require slightly more integrability at infinity than the one that one would naturally expect for the second integral in~\eqref{eq:self:sim} to be finite.

The only exception to this is the case $\alpha=\beta=0$ where we can use a different argument. Precisely, one can consider the Laplace transformed equation and derive a differential inequality which can be solved explicitly.

As a first step towards the precise asymptotic behaviour of self-similar solutions we establish a bound on the integral of $xf(x)$ over $(0,R)$ which has the expected scaling behaviour.

\begin{proposition}\label{Prop:tail:averaged}
 Assume that $K$ satisfies~\cref{eq:Ass:ker:1,eq:Ass:ker:2} and let $f$ be a self-similar profile according to Definition~\ref{Def:profile}. Then there exists a constant $C_{f}>0$ such that it holds
 \begin{equation*}
  \int_{0}^{R}yf(y)\dy\leq C_{f}R^{1-\rho}\qquad \text{for all }R>0.
 \end{equation*}
\end{proposition}

\begin{proposition}\label{Prop:bounded:kernel}
 Assume that $K$ satisfies~\cref{eq:Ass:ker:1,eq:Ass:ker:2} with $\alpha=\beta=0$ and let $f$ be a self-similar profile according to Definition~\ref{Def:profile} but satisfying~\eqref{eq:min:moment:bd} with $\gamma=0$. Then there exists a constant $C_{f}>0$ such that it holds
 \begin{equation*}
  \int_{0}^{R}yf(y)\dy\leq C_{f}R^{1-\rho}\qquad \text{for all }R>0.
 \end{equation*}
\end{proposition}

With these estimates it will be rather straightforward to derive the asymptotic behaviour of $\int_{0}^{R}xf(x)\dx$ which is already a weak form of the desired statement.

\begin{proposition}\label{Prop:asymptotics:averaged}
 Assume the same conditions as in Proposition~\ref{Prop:tail:averaged} or~\ref{Prop:bounded:kernel} respectively. Then for each self-similar profile $f$ there exists a suitable rescaling $\tilde{f}(x)\vcc=a^{1+\lambda}f(ax)$ such that 
 \begin{equation*}
  \int_{0}^{R}y\tilde{f}(y)\dy\leq R^{1-\rho} \qquad \text{and}\qquad \int_{0}^{R}y\tilde{f}(y)\dy\sim R^{1-\rho}\qquad \text{for }R\longrightarrow \infty.
 \end{equation*}
 Note that $\tilde{f}$ is also a self-similar profile according to Remark~\ref{Rem:rescaling}.
\end{proposition}

The main result of this work is the following statement which gives the precise asymptotic behaviour of self-similar profiles at infinity.

\begin{theorem}\label{Thm:asymptotics}
 Assume that $K$ satisfies~\cref{eq:Ass:ker:1,eq:Ass:ker:2} and let $f$ be a self-similar profile according to Definition~\ref{Def:profile} and subject to the normalisation given by Proposition~\ref{Prop:asymptotics:averaged}. Then $f$ is continuous on $(0,\infty)$ and satisfies $f(x)\sim (1-\rho)x^{-1-\rho}$ for $x\to \infty$.
\end{theorem}

\begin{remark}
 We note that Theorem~\ref{Thm:asymptotics} in the special case of homogeneity $\lambda=0$, i.e.\@ $\alpha=-\beta$ with $\beta\in[0,1)$ has already been treated in~\cite{Thr16} using the same kind of arguments. In this work we will however generalise the result to a much broader class of kernels which may have different homogeneity.
\end{remark}

\subsection{Outline of the article}

The remainder of this work is concerned with the proofs of these main results and as a first step, we will derive in Section~\ref{Sec:reg:zero} an estimate on the average of self-similar profiles which gives enough integrability in the region close to zero for what follows. 

In Section~\ref{Sec:averaged:tail:estimates} we will provide the proofs of~\cref{Prop:tail:averaged,Prop:bounded:kernel,Prop:asymptotics:averaged}. More precisely, in the case of Proposition~\ref{Prop:tail:averaged} the general strategy is to test the equation by the function $x^{\rho-2}$ and conclude by some iteration argument. 

Proposition~\ref{Prop:bounded:kernel} uses slightly weaker assumptions on the integrability of self-similar profiles at infinity, which is possible since for bounded kernels a completely different argument applies. More precisely, we take the Laplace transform of~\eqref{eq:self:sim} and use the boundedness of the kernel to derive a differential inequality for the Laplace transformed self-similar profile. Solving this inequality explicitly, gives an upper bound on the Laplace transform from which the desired estimate can then be derived easily.

The proof of Proposition~\ref{Prop:asymptotics:averaged} is then an easy consequence of~\cref{Prop:tail:averaged,Prop:bounded:kernel}.

Section~\ref{Sec:asymptotic:tail:behaviour} is finally devoted to the proof of Theorem~\ref{Thm:asymptotics}. The key for this will be Proposition~\ref{Prop:pointwise:decay} which states that for large values of $x$ each self-similar profile $f$ is at least bounded as $f(x)\leq Cx^{-1-\rho}$. Once this is established, one easily deduces that the non-linear term in~\eqref{eq:self:sim} is of lower order for $x\to\infty$ and the expected decay behaviour follows directly. However, the proof of Proposition~\ref{Prop:pointwise:decay} requires again some iteration argument and we have to consider additionally three different cases depending on the signs of the exponents $\alpha$ and $\beta$ although the general argument in all cases is similar.

\section{Regularity close to zero}\label{Sec:reg:zero}

In this section we will prove the following lemma which provides a uniform estimate for self-similar profiles in the region close to zero.

\begin{lemma}\label{Lem:zero:averaged}
  Assume that $K$ satisfies~\cref{eq:Ass:ker:1,eq:Ass:ker:2}. There exists a constant $C>0$ which only depends on $c_{*}$, $b$ and $B$ such that each self-similar profile satisfies
  \begin{equation*}
   \int_{R}^{2R}xf(x)\dx\leq CR^{1-\lambda} \qquad \text{for all }R>0.
  \end{equation*} 
\end{lemma}

\begin{remark}
 This statement gives that each self-similar profile behaves at least in an averaged sense not worse than $x^{-1-\lambda}$ close to zero. More precisely this is also true for large values of $x$ while it will turn out that we can obtain better decay there.
\end{remark}

The same kind of result has already been used in the case of finite mass in~\cite{NiV11}. However, the same proof also applies in the present situation where we consider fat-tailed profiles since it only relies on the structure of the non-linear term in~\eqref{eq:self:sim}. However, since we use slightly different assumptions on the kernel $K$ here and for completeness we briefly recall the argument of~\cite{NiV11}.

\begin{proof}[Proof of Lemma~\ref{Lem:zero:averaged}]
 We may assume without loss of generality that it holds $D\leq 2 d$ in the assumption~\eqref{eq:Ass:ker:2} since for $D>2d$ the lower bound in~\eqref{eq:Ass:ker:2} still holds with $D$ replaced by $2d$.
 
 We then divide~\eqref{eq:self:sim} by $x$, integrate over $[bR,BR]$ for any $R>0$ and use the non-negativity of $f$ to get
 \begin{multline*}
   \int_{bR}^{BR}xf(x)\dx\\*
   =(1-\rho)\int_{bR}^{BR}\frac{1}{x}\int_{0}^{x}yf(y)\dy\dx+\int_{bR}^{BR}\frac{1}{x}\int_{0}^{x}\int_{x-y}^{\infty}yK(y,z)f(y)f(z)\dz\dy\dx\\*
   \geq \frac{1}{BR}\int_{bR}^{BR}\int_{0}^{x}\int_{x-y}^{\infty}yK(y,z)f(y)f(z)\dz\dy\dx.
  \end{multline*}
With Fubini's Theorem and the non-negativity of the integrand, we can rewrite and estimate the right-hand side further as
 \begin{multline*}
   \int_{bR}^{BR}\int_{0}^{x}\int_{x-y}^{\infty}(\cdots)\dz\dy\dx\\*
    =\int_{0}^{bR}\int_{bR}^{BR}\int_{x-y}^{\infty}(\cdots)\dz\dx\dy+\int_{bR}^{BR}\int_{y}^{BR}\int_{x-y}^{\infty}(\cdots)\dz\dx\dy\\*
   \geq \int_{bR}^{BR}\int_{y}^{BR}\int_{x-y}^{\infty}(\cdots)\dz\dx\dy.
  \end{multline*}
Using Fubini's Theorem again to rewrite the right-hand side we can estimate once more to get
\begin{multline*}
   \int_{bR}^{BR}\int_{0}^{x}\int_{x-y}^{\infty}(\cdots)\dz\dy\dx\\*
   \geq \int_{bR}^{BR}\int_{0}^{BR-y}\int_{y}^{y+z}(\cdots)\dx\dz\dy+\int_{bR}^{BR}\int_{BR-y}^{\infty}\int_{y}^{BR}(\cdots)\dx\dz\dy\\*
  \geq \int_{bR}^{BR}\int_{BR-y}^{\infty}\int_{y}^{BR}(\cdots)\dx\dz\dy.
\end{multline*}
Reducing the domain of integration in the variables $y$ and $z$ and using the homogeneity of the kernel $K$ we thus find
\begin{equation*}
 \begin{split}
  \int_{bR}^{BR}xf(x)\dx&\geq \frac{1}{BR}\int_{bR}^{BR}\int_{BR-y}^{\infty}\int_{y}^{BR}\dx yK(y,z)f(y)f(z)\dz\dy\\
  &\geq \frac{1}{BR}\int_{bR}^{\frac{b+B}{2}R}\int_{(B-b)R}^{BR}(BR-y)yK\Bigl(R\frac{y}{R},R\frac{z}{R}\Bigr)f(y)\frac{1}{z}zf(z)\dz\dy\\
  &\geq\frac{c_{*}(B-b)}{B^2}R^{\lambda-1}\int_{bR}^{\frac{b+B}{2}R}yf(y)\dy\int_{(B-b)R}^{BR}zf(z)\dz.
 \end{split}
\end{equation*}
The assumption $B\leq 2b$ yields $\int_{(B-b)R}^{BR}zf(z)\dz\geq \int_{bR}^{BR}zf(z)\dz$ such that it follows
\begin{equation*}
 \int_{bR}^{\frac{b+B}{2}R}yf(y)\dy\leq \frac{B^2}{c_{*}(B-b)}R^{1-\lambda}.
\end{equation*}
If we replace $R$ by $R/b$ this reads as
\begin{equation}\label{eq:zero:averaged:1}
 \int_{R}^{\frac{b+B}{2b}R}yf(y)\dy\leq \frac{B^2}{c_{*}(B-b)b^{1-\lambda}}R^{1-\lambda}.
\end{equation}
To conclude the proof, we note that $(b+B)/(2b)>1$. Therefore, we can fix a constant $N\in\N$ such that it holds $\bigl(\frac{b+B}{2b}\bigr)^{N+1}\geq 2$. Splitting the integral and using~\eqref{eq:zero:averaged:1} we then obtain
\begin{equation*}
 \int_{R}^{2R}xf(x)\dx\leq \sum_{n=0}^{N}\int_{\left(\frac{b+B}{2b}\right)^{n}R}^{\left(\frac{b+B}{2b}\right)^{n+1}R}xf(x)\dx\leq \frac{B^2}{c_{*}(B-b)b^{1-\lambda}}\sum_{n=0}^{N}\left(\frac{b+B}{2b}\right)^{n} R^{1-\lambda},
\end{equation*}
which yields the claim.
\end{proof}

From Lemma~\ref{Lem:zero:averaged} we directly infer the following result which provides estimates on the integral of $f$ in the region close to zero.

\begin{lemma}\label{Lem:moment:origin}
 Let $\gamma>\lambda$. There exists a constant $C_{\gamma}>0$ such that it holds
 \begin{equation*}
  \int_{0}^{1}x^{\gamma}f(x)\dx<C
 \end{equation*}
 for each self-similar profile $f$.
\end{lemma}

\begin{proof}
 The proof follows immediately from Lemma~\ref{Lem:zero:averaged} together with a dyadic decomposition similarly as in the proof of Lemma~\ref{Lem:moment:estimates}.
\end{proof}

\section{Averaged tail estimates}\label{Sec:averaged:tail:estimates}

\subsection{Proof of Proposition~\ref{Prop:bounded:kernel}}

We start with the proof of Proposition~\ref{Prop:bounded:kernel} which relies on the Laplace transform and is much easier than the corresponding proof of Proposition~\ref{Prop:tail:averaged}. The same argument is already contained in~\cite{Thr16} while similar ideas have also been used in~\cite{NiV14a}.

\begin{proof}[Proof of Proposition~\ref{Prop:bounded:kernel}]
 Since we have $\alpha=\beta=0$ it follows from~\eqref{eq:Ass:ker:2} that $K(x,y)\leq 2C_{*}$ for all $x,y>0$. For $q\geq 0$ we then denote by
 \begin{equation*}
  Q(q)\vcc=\int_{0}^{\infty}(1-\ee^{-qx})f(x)\dx
 \end{equation*}
 the desingularised Laplace transform of the self-similar profile $f$. Note that $Q$ is well-defined due to the condition~\eqref{eq:min:moment:bd} (with $\gamma=0$) and Lemma~\ref{Lem:moment:origin}. Moreover, $Q$ as a function of $q$ is infinitely often differentiable on $(0,\infty)$ and we have in particular
 \begin{equation}\label{eq:rel:des:Lap:trans}
  Q'(q)=\int_{0}^{\infty}xf(x)\ee^{-qx}\dx\qquad \text{and}\qquad Q''(q)=-\int_{0}^{\infty}x^2f(x)\ee^{-qx}\dx.
 \end{equation}
  We next multiply equation~\eqref{eq:self:sim} by $\ee^{-qx}$ and integrate over $(0,\infty)$ to obtain
  \begin{multline*}
   \int_{0}^{\infty}x^2f(x)\ee^{-qx}\dx=(1-\rho)\int_{0}^{\infty}\ee^{-qx}\int_{0}^{x}yf(y)\dy\dx\\
   +\int_{0}^{\infty}\ee^{-qx}\int_{0}^{x}\int_{x-y}^{\infty}yK(y,z)f(y)f(z)\dz\dy\dx.
  \end{multline*}
 Applying Fubini's Theorem on the right-hand side and using the relations given in~\eqref{eq:rel:des:Lap:trans} we obtain after some elementary rearrangements that $Q$ satisfies the equation
 \begin{multline}\label{eq:des:self:sim}
  -\del_{q}\bigl(qQ''(q)\bigr)=-\rho Q'(q)\\*
  +\del_{q}\biggl(\frac{1}{2}\int_{0}^{\infty}\int_{0}^{\infty}K(y,z)f(y)f(z)(1-\ee^{-qy})(1-\ee^{-qz})\dz\dy\biggr).
 \end{multline}
We first note that the second term on the right-hand side is well-defined due to~\eqref{eq:min:moment:bd} and Lemma~\ref{Lem:moment:origin} and moreover it follows by monotone convergence that
\begin{equation*}
 \frac{1}{2}\int_{0}^{\infty}\int_{0}^{\infty}K(y,z)f(y)f(z)(1-\ee^{-qy})(1-\ee^{-qz})\dz\dy\longrightarrow 0\quad \text{for } q\longrightarrow 0.
\end{equation*}
Similarly, we have
\begin{equation*}
 Q(q)\longrightarrow 0\quad \text{and}\quad qQ'(q)\longrightarrow 0\qquad \text{for }q\longrightarrow 0.
\end{equation*}
As a consequence, we can integrate equation~\eqref{eq:des:self:sim} over $(0,q)$ to get
\begin{equation}\label{eq:Q:1}
 -qQ'(q)=-\rho Q(q)+\frac{1}{2}\int_{0}^{\infty}\int_{0}^{\infty}K(y,z)f(y)f(z)(1-\ee^{-qy})(1-\ee^{-qz})\dz\dy.
\end{equation}
From the non-negativity of $f$ and the bound $K(y,z)\leq 2C_{*}$ we can thus derive the inequality
\begin{equation*}
 -qQ'(q)\leq -\rho Q(q)+C_{*}Q^2(q).
\end{equation*}
Since $f$ is non-negative and non-trivial it holds that $Q(q)>0$ for all $q>0$ and thus we can rearrange the inequality to get
\begin{equation*}
 \del_{q}\biggl(q^{\rho}\Bigl(\bigl(Q(q)\bigr)^{-1}-\frac{C_{*}}{\rho}\Bigr)\biggr)\leq 0.
\end{equation*}
Due to the fact that $Q(q)\to 0$ for $q\to 0$ we can fix $\widehat{q}>0$ such that $(1/Q(q)-C_{*}/\rho)>0$ for all $q\leq \widehat{q}$. If we then fix the constant $\widehat{C}\vcc=\widehat{q}^{\;\rho}(1/Q(\,\widehat{q}\,)+C_{*}/\rho)>0$ and integrate over $[\,q,\widehat{q}\,]$ we obtain after some rearrangement that
\begin{equation*}
 Q(q)\leq \frac{q^{\rho}}{\frac{C_{*}}{\rho}q^{\rho}+\widehat{C}^{-1}}\leq \widehat{C}q^{\rho} \qquad \text{for all }q\leq \widehat{q}.
\end{equation*}
If we use this estimate together with~\eqref{eq:Q:1} and the non-negativity of $f$ we find
\begin{equation*}
 Q'(q)\leq \frac{\rho}{q}Q(q)\leq \widehat{C}\rho q^{\rho-1} \quad \text{for all }q\leq \widehat{q}.
\end{equation*}
Now, if we recall also~\eqref{eq:rel:des:Lap:trans} we find for all $R\geq \widehat{q}^{\,-1}$ that it holds
\begin{equation*}
 \int_{R}^{2R}xf(x)\dx=\int_{R}^{2R}xf(x)\ee^{-\frac{x}{R}}\ee^{\frac{x}{R}}\dx\leq \ee^{2}\int_{0}^{\infty}xf(x)\ee^{-\frac{x}{R}}\dx\leq \widehat{C}\rho\ee^{2}R^{1-\rho}.
\end{equation*}
Together with Lemma~\ref{Lem:zero:averaged} the claim then directly follows by means of a dyadic decomposition similarly as in the proof of Lemma~\ref{Lem:moment:estimates}.
\end{proof}

The following elementary lemma will be used frequently in the proof of Proposition~\ref{Prop:tail:averaged}.

\begin{lemma}\label{Lem:inner:integral}
 For each $\rho\in(\lambda,1)$ there exists a constant $C_{\rho}$ such that it holds
 \begin{equation*}
  \int_{y}^{y+z}x^{\rho-2}\dx\leq C_{\rho}y^{\rho-1-a}z^{a}
 \end{equation*}
 for all $y,z>0$ and all $a\in[0,1]$.
\end{lemma}

\begin{proof}
 For $y\leq z$ we get the estimate
 \begin{equation*}
  \begin{split}
   \int_{y}^{y+z}x^{\rho-2}\dx=\frac{1}{1-\rho}\bigl(y^{\rho-1}-(y+z)^{\rho-1}\bigr)\leq \frac{1}{1-\rho}y^{\rho-1-a}y^{a}\leq \frac{1}{1-\rho}y^{\rho-1-a}z^a.
  \end{split}
 \end{equation*}
 On the other hand, if $z\leq y$, we find
 \begin{equation*}
  \int_{y}^{y+z}x^{\rho-2}\dx\leq y^{\rho-2}z=y^{\rho-2}z^{1-a}z^{a}\leq y^{\rho-a-1}z^{a}.
 \end{equation*}
 Thus, the claim follows with $C_{\rho}\vcc=\max\{1,1/(1-\rho)\}$.
\end{proof}

\subsection{Proof of Proposition~\ref{Prop:tail:averaged}}

We will now give the proof of Proposition~\ref{Prop:tail:averaged} which relies essentially on an iterated moment estimate.

\begin{proof}[Proof of Proposition~\ref{Prop:tail:averaged}]
 Note first that in view of Lemma~\ref{Lem:zero:averaged} it suffices to show the claim for $R\geq 1$. We then multiply equation~\eqref{eq:self:sim} by $x^{\rho-2}$ and integrate over $(0,R)$ to get
 \begin{multline*}
  \int_{0}^{R}x^{\rho}f(x)\dx+(\rho-1)\int_{0}^{R}x^{\rho-2}\int_{0}^{x}yf(y)\dy\dx\\*
  =\int_{0}^{R}x^{\rho-2}\int_{0}^{x}\int_{x-y}^{\infty}yK(y,z)f(y)f(z)\dz\dy\dx.
 \end{multline*}
We next apply Fubini's Theorem which is possible as the following estimates will show. This then yields
\begin{multline*}
 \int_{0}^{R}x^{\rho}f(x)\dx+(\rho-1)\int_{0}^{R}yf(y)\int_{y}^{R}x^{\rho-2}\dx\dy\\*
  =\int_{0}^{R}\int_{y}^{R}\int_{x-y}^{\infty}x^{\rho-2}yK(y,z)f(y)f(z)\dz\dx\dy.
\end{multline*}
Due to the non-negativity we can enlarge the domain of integration on the right-hand side which together with the evaluation of the integral $\int_{y}^{R}x^{\rho-2}\dx$ gives the estimate
\begin{equation*}
  R^{\rho-1}\int_{0}^{R}yf(y)\dy\leq \int_{0}^{R}\int_{y}^{\infty}\int_{x-y}^{\infty}x^{\rho-2}yK(y,z)f(y)f(z)\dz\dx\dy.
\end{equation*}
After another application of Fubini's Theorem we finally arrive at
\begin{equation}\label{eq:pos:exp:1}
 R^{\rho-1}\int_{0}^{R}yf(y)\dy\leq \int_{0}^{R}\int_{0}^{\infty}yK(y,z)f(y)f(z)\int_{y}^{y+z}x^{\rho-2}\dx\dz\dy.
\end{equation}
 In order to proceed we have to derive an estimate on the integral on the right-hand side which requires to consider the regions close and far from zero separately. We thus split the integral $\int_{0}^{R}\int_{0}^{\infty}(\cdots)\dz\dy$ as
\begin{multline}\label{eq:pos:exp:2}
  \int_{0}^{R}\int_{0}^{\infty}(\cdots)\dz\dy=\int_{0}^{1}\int_{0}^{1}(\cdots)\dz\dy+\int_{1}^{R}\int_{0}^{1}(\cdots)\dz\dy\\*
  +\int_{0}^{1}\int_{1}^{\infty}(\cdots)\dz\dy+\int_{1}^{R}\int_{1}^{\infty}(\cdots)\dz\dy.
\end{multline}
We will now estimate the four integrals on the right-hand side separately and we start with the first one. For this, we use the upper bound on $K$ which is given by~\eqref{eq:Ass:ker:2} together with Lemma~\ref{Lem:inner:integral} one time with $a_{1}=\max\{(\alpha+\rho-\beta)/2,0\}$ and one time with $a_{2}=\max\{(\beta+\rho-\alpha)/2,0\}$ which yields
\begin{multline*}
  \int_{0}^{1}\int_{0}^{1}yK(y,z)f(y)f(z)\int_{y}^{y+z}x^{\rho-2}\dx\dz\dy\\*
  \leq C\int_{0}^{1}y^{1+\alpha}f(y)\int_{0}^{1}z^{\beta}f(z)y^{\rho-a_{1}-1}z^{a_{1}}\dz\dy\\*
  +C\int_{0}^{1}y^{1+\beta}f(y)\int_{0}^{1}z^{\alpha}f(z)y^{\rho-a_{2}-1}z^{a_{2}}\dz\dy.
 \end{multline*}
If we plug in the values of $a_{1}$ and $a_{2}$ it follows together with $\lambda=\alpha+\beta$ and Lemma~\ref{Lem:moment:origin} that
\begin{multline}\label{eq:pos:exp:3}
 \int_{0}^{1}\int_{0}^{1}yK(y,z)f(y)f(z)\int_{y}^{y+z}x^{\rho-2}\dx\dz\dy\\*
 \leq C\int_{0}^{1}y^{\max\{(\lambda+\rho)/2,\alpha+\rho\}}f(y)\dy\int_{0}^{1}z^{\max\{(\lambda+\rho)/2,\beta\}}f(z)\dz\\*
 +C\int_{0}^{1}y^{\max\{(\lambda+\rho)/2,\beta+\rho\}}f(y)\dy\int_{0}^{1}z^{\max\{(\lambda+\rho)/2,\alpha\}}f(z)\dz\leq C.
\end{multline}
Note that it holds $(\lambda+\rho)/2>\lambda$ since we have $\rho>\lambda$ by assumption. We next consider the second integral on the right-hand side of~\eqref{eq:pos:exp:2} for which we again use the upper bound on $K$ given by~\eqref{eq:Ass:ker:2} as well as Lemma~\ref{Lem:inner:integral} with $a=1$. This then yields 
\begin{multline*}
  \int_{1}^{R}\int_{0}^{1}yK(y,z)f(y)f(z)\int_{y}^{y+z}x^{\rho-2}\dx\dz\dy\\*
  \leq C\int_{1}^{\infty}\int_{0}^{1}\bigl(y^{\alpha+\rho-1}z^{1+\beta}+y^{\beta+\rho-1}z^{1+\alpha}\bigr)f(y)f(z)\dz\dy.
 \end{multline*}
We now use that $y^{\rho-1}\leq 1$ if $y\geq 1$ which gives together with Lemma~\ref{Lem:moment:origin} and~\eqref{eq:min:moment:bd} that
\begin{equation}\label{eq:pos:exp:4}
 \begin{split}
  &\hphantom{{}\leq{}}\int_{1}^{R}\int_{0}^{1}yK(y,z)f(y)f(z)\int_{y}^{y+z}x^{\rho-2}\dx\dz\dy\\
  &\leq C\int_{1}^{\infty}y^{\alpha}f(y)\dy\int_{0}^{1}z^{1+\beta}f(z)\dz+C\int_{1}^{\infty}y^{\beta}f(y)\dy\int_{0}^{1}z^{1+\alpha}f(z)\dz\leq C.
 \end{split}
\end{equation}
The third integral on the right-hand side of~\eqref{eq:pos:exp:2} can be estimated similarly. Precisely, if we use the upper bound~\eqref{eq:Ass:ker:2} on $K$ and Lemma~\ref{Lem:inner:integral} with $a=0$ together with Lemma~\ref{Lem:moment:origin} and~\eqref{eq:min:moment:bd} we get
\begin{multline}\label{eq:pos:exp:5}
  \int_{0}^{1}\int_{1}^{\infty}yK(y,z)f(y)f(z)\int_{y}^{y+z}x^{\rho-2}\dx\dz\dy\\*
  \leq C\int_{0}^{1}\int_{1}^{\infty}\bigl(y^{\alpha+\rho}z^{\beta}+y^{\beta+\rho}z^{\alpha}\bigr)f(y)f(z)\dz\dy\leq C.
\end{multline}
It thus remains to estimate the fourth integral on the right-hand side of~\eqref{eq:pos:exp:2} for which we proceed similarly, i.e.\@ we fix $\gamma>\lambda$ according to~\eqref{eq:min:moment:bd}, use~\eqref{eq:Ass:ker:2} and apply Lemma~\ref{Lem:inner:integral} once with $a=\gamma-\beta$ and once with $a=\gamma-\alpha$. Since $\gamma\geq \beta\geq\alpha$, this then yields
\begin{multline*}
 \int_{1}^{R}\int_{1}^{\infty}yK(y,z)f(y)f(z)\int_{y}^{y+z}x^{\rho-2}\dx\dz\dy\\*
 \leq C\int_{1}^{R}y^{1+\alpha}f(y)\int_{1}^{\infty}z^{\beta}f(z)y^{\rho-\gamma+\beta-1}z^{\gamma-\beta}\dz\dy\\*
 +C\int_{1}^{R}y^{1+\beta}f(y)\int_{1}^{\infty}z^{\alpha}f(z)y^{\rho-\gamma+\alpha-1}z^{\gamma-\alpha}\dz\dy.
\end{multline*}
Together with~\eqref{eq:min:moment:bd} we get
\begin{multline}\label{eq:pos:exp:6}
 \int_{1}^{R}\int_{1}^{\infty}yK(y,z)f(y)f(z)\int_{y}^{y+z}x^{\rho-2}\dx\dz\dy\\*
 \leq C\int_{1}^{R}y^{\lambda+\rho-\gamma}f(y)\dy\int_{1}^{\infty}z^{\gamma}f(z)\dz\leq C\int_{1}^{R}y^{\lambda+\rho-\gamma}f(y)\dy.
\end{multline}
Thus, if we summarise~\cref{eq:pos:exp:1,eq:pos:exp:2,eq:pos:exp:3,eq:pos:exp:4,eq:pos:exp:5,eq:pos:exp:6} it follows
\begin{equation}\label{eq:pos:exp:7}
 R^{\rho-1}\int_{0}^{R}yf(y)\dy\leq C+C\int_{1}^{R}y^{\lambda+\rho-\gamma}f(y)\dy.
\end{equation}
Now, we have to distinguish the two cases $\lambda+\rho-\gamma\leq \gamma$ or equivalently $\gamma\geq (\lambda+\rho)/2$ and $\lambda+\rho-\gamma>\gamma$ which means $\gamma<(\lambda+\rho)/2$. In the first case, we can argue directly since due to~\eqref{eq:min:moment:bd} the right-hand side of~\eqref{eq:pos:exp:7} is bounded and we thus find
\begin{equation*}
 R^{\rho-1}\int_{0}^{R}yf(y)\dy\leq C \quad \text{or equivalently}\quad \int_{0}^{R}yf(y)\dy\leq CR^{1-\rho},
\end{equation*}
which shows the claim in this case. On the other hand, if $\lambda+\rho-\gamma>\gamma$ we have to iterate the previous estimate. Precisely, we can further estimate the right-hand side of~\eqref{eq:pos:exp:7} together with~\eqref{eq:min:moment:bd} to get
\begin{equation*}
 R^{\rho-1}\int_{0}^{R}yf(y)\dy\leq C+C\int_{1}^{R}y^{\lambda+\rho-2\gamma}y^{\gamma}f(y)\dy\leq C+CR^{\lambda+\rho-2\gamma}.
\end{equation*}
In the last step we used that $\lambda+\rho-\gamma>\gamma$, i.e.\@ $\lambda+\rho-2\gamma>0$. Since we also assume $R\geq 1$, the first constant on the right-hand side can be estimated by the second term and we further conclude
\begin{equation*}
 \int_{0}^{R}yf(y)\dy\leq CR^{\lambda-2\gamma-1}\qquad \text{for }R\geq 1.
\end{equation*}
A dyadic argument then shows that
\begin{equation}\label{eq:pos:exp:8}
 \int_{1}^{\infty}x^{\nu}f(x)\dx\leq C_{\nu} \qquad \text{for all }\nu<2\gamma-\lambda.
\end{equation}
If we recall~\eqref{eq:pos:exp:7} we again have to distinguish two cases, namely $\lambda+\rho-2\gamma<2\gamma-\lambda$ and $\lambda+\rho-2\gamma\geq 2\gamma-\lambda$. If the first inequality holds, the claim follows since then the right-hand side of~\eqref{eq:pos:exp:7} is uniformly bounded due to~\eqref{eq:pos:exp:8}. On the other hand, if the second case applies, we have to iterate further. Precisely, due to~\eqref{eq:pos:exp:8} and Lemma~\ref{Lem:moment:origin} we obtain from~\eqref{eq:pos:exp:7} for each $\gamma_{1}<2\gamma-\lambda$ that
\begin{equation*}
 R^{\rho-1}\int_{0}^{R}yf(y)\dx\leq C+\int_{1}^{R}y^{\lambda+\rho-\gamma-\gamma_{1}}y^{\gamma_{1}}f(y)\dy\leq C_{\gamma_{1}}R^{\lambda+\rho-\gamma-\gamma_{1}} \quad \text{for }R\geq 1.
\end{equation*}
As before, a dyadic argument thus gives
\begin{equation*}
 \int_{1}^{\infty}x^{\nu}f(x)\dx\leq C_{\nu} \quad \text{for all } \nu<3\gamma-2\lambda.
\end{equation*}
If we continue this procedure iteratively, we obtain after a finite number of steps that it holds
\begin{equation*}
 \int_{1}^{\infty}x^{\nu}f(x)\dx\leq C_{\nu} \quad \text{for all }\nu<\gamma+n(\gamma-\lambda)
\end{equation*}
for some $n\in\N$ with $\lambda+\rho-2\gamma-n(\gamma-\lambda)<0$. Then the claim follows as above, since the right-hand side of~\eqref{eq:pos:exp:7} can be bounded uniformly with respect to $R$.
\end{proof}

\subsection{Proof of Proposition~\ref{Prop:asymptotics:averaged}}

We are now prepared to give the proof of Proposition~\ref{Prop:asymptotics:averaged}.

\begin{proof}[Proof of Proposition~\ref{Prop:asymptotics:averaged}]
 According to \cref{Prop:tail:averaged,Prop:bounded:kernel} there exists for each self-similar profile $f$ a constant $C_{f}>0$ such that it holds $\int_{0}^{R}xf(x)\dx\leq C_{f}R^{1-\rho}$ for all $R>0$. 
 
 For a given self-similar profile $f$ we define the function $M_{f}(R)\vcc=\int_{0}^{R}xf(x)\dx$ and note that due to the non-negativity of $f$ it follows immediately from~\eqref{eq:self:sim} that 
 \begin{equation*}
  M_{f}'(R)\geq \frac{1-\rho}{R}M(R) \quad \text{or equivalently}\quad \del_{R}\bigl(R^{\rho-1}M_{f}(R)\bigr)\geq 0.
 \end{equation*}
 Thus, the function $R^{\rho-1}M_{f}(R)$ is non-decreasing and according to \cref{Prop:tail:averaged,Prop:bounded:kernel} we also have $R^{\rho-1}M_{f}(R)\leq C_{f}$ uniformly with respect to $R$.
 
 If we consider now the rescaling $\tilde{f}(x)=a^{1+\lambda}f(ax)$ one directly verifies that $\tilde{f}$ satisfies $\int_{0}^{R}x\tilde{f}(x)\dx\leq C_{f}a^{\lambda-\rho}R^{1-\rho}$ for all $R>0$ and all $a>0$, i.e.\@ $R^{\rho-1}M_{\tilde{f}}(R)\leq a^{\lambda-\rho}C_{f}$. Due to the monotonicity of $R^{\rho-1}M_{\tilde{f}}(R)$, we can thus fix the parameter $a>0$ such that it holds
 \begin{equation*}
  \lim_{R\to\infty}R^{\rho-1}M_{\tilde{f}}(R)=1.
 \end{equation*}
 This then already shows that $\int_{0}^{R}y\tilde{f}(y)\dy\sim R^{1-\rho}$ for $R\to\infty$. On the other hand, we have by monotonicity that $R^{\rho-1}M_{\tilde{f}}(R)\leq N^{\rho-1}M_{\tilde{f}}(N)$ for each $N\geq R$. Thus, taking the limit $N\to\infty$ together with the choice of $a$ we get $R^{\rho-1}M_{\tilde{f}}(R)\leq 1$ or equivalently $\int_{0}^{R}y\tilde{f}(y)\dy\leq R^{1-\rho}$, which finishes the proof. 
\end{proof}

\begin{remark}\label{Rem:rescaling}
 From now on we normalise all self-similar profiles according to Proposition~\ref{Prop:asymptotics:averaged}.
\end{remark}

\subsection{Moment estimates}
 
 As a consequence of Proposition~\ref{Prop:tail:averaged} and Lemma~\ref{Lem:zero:averaged} we provide now several moment estimates which we will use frequently in the following.
 
 \begin{lemma}\label{Lem:moment:estimates}
  Let $f$ be a self-similar profile according to the normalisation given by Proposition~\ref{Prop:asymptotics:averaged} and let $\nu>0$. Then we have the estimates
  \begin{align}
   \int_{x}^{\infty}y^{\chi}f(y)\dy&\leq C x^{\chi-\rho} && \text{for all } \chi<\rho \text{ and for all } x>0,\label{eq:moment:1}\\
   \int_{x}^{\infty}y^{\chi}f(y)\dy&\leq C x^{\chi-\lambda} && \text{for all } \chi<\lambda  \text{ and for all } x>0,\label{eq:moment:1.5}\\
   \int_{0}^{x}y^{\chi}f(y)\dy&\leq Cx^{\chi-\lambda} && \text{for all } \chi>\lambda  \text{ and for all } x>0,\label{eq:moment:2}\\
   \int_{0}^{x}y^{\chi}f(y)\dy&\leq Cx^{\max\{\chi-\rho,1-\rho\}} && \text{for all } \chi>\lambda  \text{ and for all } x\geq 1,\label{eq:moment:2.5}\\
   \int_{0}^{x}y^{\chi}f(y)\dy&\leq C_{\nu}x^{\max\{\chi-\rho+\nu,0\}} && \text{for all } \chi>\lambda  \text{ and for all } x\geq 1,\label{eq:moment:3}\\
   \int_{0}^{\infty}y^{\chi}f(y)\dy&<C && \text{for all } \chi\in(\lambda,\rho)\label{eq:moment:4}.
  \end{align}
   These estimates remain true also without the rescaling of Proposition~\ref{Prop:asymptotics:averaged} while then the constants depend on $f$.
 \end{lemma}

 \begin{proof}
 The estimates~\cref{eq:moment:1,eq:moment:1.5,eq:moment:2,eq:moment:3,eq:moment:4} follow easily from Proposition~\ref{Prop:asymptotics:averaged} and Lemma~\ref{Lem:zero:averaged} together with a dyadic decomposition. The same kind of argument has already been used several times (e.g.\@ \cite{NiV13a,NTV16a}) but to recall the idea we show~\eqref{eq:moment:1}. Precisely, we split the integral $\int_{x}^{\infty}(\cdots)\dy$ and use Proposition~\ref{Prop:asymptotics:averaged} to get
  \begin{multline*}
    \int_{x}^{\infty}y^{\chi}f(y)\dy=\sum_{n=0}^{\infty}\int_{2^{n}x}^{2^{n+1}x}y^{\chi-1}yf(y)\dy\\*
    \leq 2^{\max\{\chi-1,0\}}\sum_{n=0}^{\infty}2^{(\chi-1)n}x^{\chi-1}\int_{0}^{2^{n+1}x}yf(y)\dy\\*
    \leq 2^{\max\{\chi-\rho,1-\rho\}}x^{\chi-\rho}\sum_{n=0}^{\infty}2^{(\chi-\rho)n}=Cx^{\chi-\rho}.
   \end{multline*}
 The sum in the last step is finite due to the condition $\chi<\rho$. Using Lemma~\ref{Lem:zero:averaged} instead of Proposition~\ref{Prop:asymptotics:averaged} we similarly obtain~\cref{eq:moment:1.5,eq:moment:2}. The estimate~\eqref{eq:moment:2.5} follows from~\eqref{eq:moment:2} together with Proposition~\ref{Prop:asymptotics:averaged}. Precisely, for $x\geq 1$ we have
 \begin{multline*}
   \int_{0}^{x}y^{\chi}f(y)\dy=\int_{0}^{1}y^{\chi}f(y)\dy+\int_{1}^{x}y^{\chi-1}yf(y)\dy\\*
   \leq C+Cx^{\max\{\chi-1,0\}}x^{1-\rho}\leq Cx^{\max\{\chi-\rho,1-\rho\}}.
  \end{multline*}
 In the last step we additionally used that $1\leq x^{\max\{\chi-\rho,1-\rho\}}$ for $x\geq 1$ since the exponent is non-negative. Similarly, we deduce from~\cref{eq:moment:1,eq:moment:2} that
 \begin{equation*}
  \begin{split}
   \int_{0}^{x}y^{\chi}f(y)\dy&=\int_{0}^{1}y^{\chi}f(y)\dy+\int_{1}^{x}y^{\chi-\rho+\nu}y^{\rho-\nu}f(y)\dy\\
   &\leq C+x^{\max\{\chi-\rho+\nu,0\}}\int_{1}^{\infty}y^{\rho-\nu}f(y)\dy\leq C_{\nu}x^{\max\{\chi-\rho+\nu,0\}}.
  \end{split}
 \end{equation*}
  To obtain~\eqref{eq:moment:4} we just use the splitting $\int_{0}^{\infty}(\cdots)\dy=\int_{0}^{1}(\cdots)\dy+\int_{1}^{\infty}(\cdots)\dy$ together with~\cref{eq:moment:1,eq:moment:2}.
 \end{proof}

\section{Asymptotic tail behaviour}\label{Sec:asymptotic:tail:behaviour}

In this section we will provide the proof of Theorem~\ref{Thm:asymptotics}. The key for this will be the following result which states that each self-similar profile decays at least as fast as $x^{-1-\rho}$ for large values of $x$.

\begin{proposition}\label{Prop:pointwise:decay}
 Let $f$ be a self-similar profile. There exist constants $C,R_{*}>0$ such that it holds
 \begin{equation}\label{eq:decay:0}
  f(x)\leq CR^{-1-\rho}\qquad \text{if } x\in[R,2R]
 \end{equation}
 for all $R\geq R_{*}$. In particular this means that $f(x)\leq 2^{1+\rho}Cx^{-1-\rho}$ for $x\geq R_{*}$. Furthermore, there exists $\delta>0$ such that 
 \begin{equation}\label{eq:I:decay:0}
  I[f]\vcc=\int_{0}^{x}\int_{x-y}^{\infty}yK(y,z)f(y)f(z)\dz\dy\leq Cx^{1-\rho-\delta}
 \end{equation}
 if $x\geq 2R_{*}$.
\end{proposition}

The proof of this result has to be split into three parts depending on the values of the exponents $\alpha$ and $\beta$. The main ingredient for the proof of Proposition~\ref{Prop:pointwise:decay} will be the following variant of Young's inequality which can be found for example as equation~(72) in~\cite{NiV13a}.

\begin{lemma}\label{Lem:Young}
 Let $d,D>0$ be constants with $d<D$ and $p,q,r\in[1,\infty]$ with $1/p+1/q=1/r+1$. If $h\in L^{q}(d/2,D)$ and $g\in L^{p}(0,D)$ then it holds
 \begin{equation*}
  \norm[\bigg]{\int_{x/2}^{x}h(y)g(x-y)\dy}_{L^{r}(d,D)}\leq \norm{h}_{L^{q}(d/2,D)}\norm{g}_{L^{p}(0,D)}.
 \end{equation*}
\end{lemma}

In~\cite{NiV13a} this result has been used qualitatively to show that self-similar profiles are locally bounded. We will proceed here similarly but we apply this estimate in a more quantitative way, i.e.\@ we choose $d$ and $D$ as multiples of $R$ and estimate precisely the dependence on $R$ of the appearing constants.

\subsection{The case $\alpha\leq \beta<0$}\label{Sec:neg:exponents}

We first consider the situation where the parameters $\alpha$ and $\beta$ satisfy $\alpha\leq \beta<0$. This is the easiest case since the statement follows rather directly, i.e.\@ an iteration argument is not necessary.

\begin{proof}[Proof of Proposition~\ref{Prop:pointwise:decay} for $\alpha\leq \beta<0$]
 We first split the integral $I[f]$ as
 \begin{equation}\label{eq:I:splitting}
  I[f](x)=\int_{0}^{x/2}\int_{x-y}^{\infty}(\cdots)\dy\dz+\int_{x/2}^{x}\int_{x-y}^{\infty}(\cdots)\dz\dy.
 \end{equation}
Together with the upper bound on $K$ in~\eqref{eq:Ass:ker:2} this yields the estimate
 \begin{multline}\label{eq:neg:decay:1}
     I[f](x)\leq C\int_{0}^{x/2}\int_{x/2}^{\infty}\bigl(y^{1+\alpha}z^{\beta}+y^{1+\beta}z^{\alpha}\bigr)f(y)f(z)\dz\dy\\*
     +C\int_{x/2}^{x}\int_{0}^{\infty}\bigl(y^{1+\alpha}z^{\beta}+y^{1+\beta}z^{\alpha}\bigr)f(y)f(z)\dz\dy.
 \end{multline}
From Lemma~\ref{Lem:moment:estimates} (i.e.\@ \cref{eq:moment:1,eq:moment:4}) we recall that
\begin{equation}\label{eq:neg:decay:2}
 \int_{x/2}^{\infty}z^{\alpha}f(z)\dz\leq Cx^{\alpha-\rho}\qquad \text{and}\qquad \int_{x/2}^{\infty}z^{\beta}f(z)\dz\leq Cx^{\beta-\rho}
\end{equation}
as well as
\begin{equation}\label{eq:neg:decay:3}
 \int_{0}^{\infty}z^{\alpha}f(z)\dz\leq C\qquad \text{and}\qquad \int_{0}^{\infty}z^{\beta}f(z)\dz\leq C.
\end{equation}
Note that for the last two estimates it is essential that we have $\alpha,\beta<0$ to get $\alpha,\beta>\lambda$ which allows to apply~\eqref{eq:moment:4}. Using~\eqref{eq:neg:decay:2} and~\eqref{eq:neg:decay:3} in~\eqref{eq:neg:decay:1} we find
\begin{equation}\label{eq:neg:decay:4}
 I[f](x)\leq C\int_{0}^{x/2}\bigl(x^{\beta-\rho}y^{1+\alpha}+x^{\alpha-\rho}y^{1+\beta}\bigr)f(y)\dy+C\int_{x/2}^{x}\bigl(y^{1+\alpha}+y^{1+\beta}\bigr)f(y)\dy.
\end{equation}
We next note that due to $\alpha,\beta>\lambda$ we get for $x\geq 2$ from Lemma~\ref{Lem:moment:estimates} (i.e.\@ \eqref{eq:moment:2.5}) that
\begin{equation*}
 \int_{0}^{x/2}y^{1+\alpha}f(y)\dy\leq Cx^{1-\rho}\qquad \text{and}\qquad \int_{0}^{x/2}y^{1+\beta}f(y)\dy\leq Cx^{1-\rho}.
\end{equation*}
 Moreover, we have $\int_{x/2}^{x}f(y)\dy=\int_{x/2}^{x}\frac{y}{y}f(y)\dy\leq Cx^{-1}x^{1-\rho}=Cx^{-\rho}$. If we use these estimates in~\eqref{eq:neg:decay:4} it follows
\begin{equation}\label{eq:neg:decay:5}
   I[f](x)\leq C\bigl(x^{1+\beta-2\rho}+x^{1+\alpha-2\rho}+x^{1+\alpha-\rho}+x^{1+\beta-\rho}\bigr)\leq Cx^{1-\rho+\beta}.
\end{equation}
Note that in the last step we used that $x\geq 2$ and $\alpha\leq \beta$. This then already shows~\eqref{eq:I:decay:0} with $\delta=-\beta$ which is positive by the assumption on $\beta$. To conclude the proof of the Proposition for $\alpha\leq\beta<0$ we use that~\eqref{eq:self:sim} reads as
\begin{equation*}
 f(x)=(1-\rho)x^{-2}\int_{0}^{x}yf(y)\dy+x^{-2}I[f](x).
\end{equation*}
Thus, \eqref{eq:neg:decay:5} and Proposition~\ref{Prop:tail:averaged} imply
\begin{equation*}
 f(x)\leq Cx^{-1-\rho}+Cx^{-1-\rho+\beta}\leq Cx^{-1-\rho}\qquad \text{for } x\geq 2
\end{equation*}
since $\beta<0$. The estimate~\eqref{eq:decay:0} can then be derived immediately for $R_{*}=2$.
\end{proof}

\subsection{The case $0<\alpha\leq \beta$}\label{Sec:pos:exp}

We next consider the case where both exponents $\alpha$ and $\beta$ are positive and the general strategy is similar to that one in Section~\ref{Sec:neg:exponents}, i.e.\@ we split the integral $I[f]$ and estimate the terms separately. However, now we cannot argue that directly but we have to use some iteration argument based on Lemma~\ref{Lem:Young} to conclude.

\begin{proof}[Proof of Proposition~\ref{Prop:pointwise:decay} for $0<\alpha\leq \beta$]
 We recall the splitting~\eqref{eq:I:splitting} which yields together with~\eqref{eq:Ass:ker:2} that
 \begin{multline}\label{eq:decay:1}
  I[f](x)\leq C\int_{0}^{x/2}\int_{x/2}^{\infty}\bigl(y^{1+\alpha}z^{\beta}+y^{1+\beta}z^{\alpha}\bigr)f(y)f(z)\dz\dy\\*
  +C\int_{x/2}^{x}y^{1+\alpha}f(y)\int_{x-y}^{\infty}z^{\beta}f(z)\dz\dy
  +C\int_{x/2}^{x}y^{1+\beta}f(y)\int_{x-y}^{\infty}z^{\alpha}f(z)\dz\dy.
 \end{multline}
 We recall from Lemma~\ref{Lem:moment:estimates} (i.e.\@ \cref{eq:moment:1,eq:moment:1.5}) that we have for all $x>0$ that
 \begin{equation*}
  \int_{x}^{\infty}y^{\chi}f(y)\dy\leq Cx^{\chi-\rho} \quad \text{and}\quad  \int_{x-y}^{\infty}z^{\chi}f(z)\dz\leq C(x-y)^{\chi-\rho}\qquad\text{ if } \chi<\rho.
 \end{equation*}
 If we use these two estimates with $\chi=\alpha$ and $\chi=\beta$ we deduce from~\eqref{eq:decay:1} that
 \begin{multline*}
  I[f](x)\leq C\int_{0}^{x/2}\bigl(y^{1+\alpha}x^{\beta-\rho}+y^{1+\beta}x^{\alpha-\rho}\bigr)f(y)\dy\\*
  +C\int_{x/2}^{x}y^{1+\alpha}(x-y)^{-\rho}f(y)\dy+C\int_{x/2}^{x}y^{1+\beta}(x-y)^{-\rho}f(y)\dy.
 \end{multline*}
 Since $0<\alpha\leq \beta$ we can further estimate this together with Proposition~\ref{Prop:tail:averaged} to get
 \begin{equation}\label{eq:decay:2}
  I[f](x)\leq Cx^{1+\lambda-2\rho}+Cx^{1+\alpha}\int_{x/2}^{x}f(y)(x-y)^{-\rho}\dy+Cx^{1+\beta}\int_{x/2}^{x}f(y)(x-y)^{-\rho}\dy.
 \end{equation}
 We now fix two parameters $p>1$ and $n_{*}\in\N$ such that
 \begin{equation}\label{eq:pos:choice:p}
  p=1+\frac{1}{n_{*}}\quad \text{with } n_{*}\text{ large enough such that} \quad 1<p<\frac{1}{\rho}.
 \end{equation}
 In terms of $I[f]$ equation~\eqref{eq:self:sim} reads as
 \begin{equation*}
  f(x)=(1-\rho)x^{-2}\int_{0}^{x}yf(y)\dy+x^{-2}I[f](x).
 \end{equation*}
 Thus, for $n<n_{*}$ we can estimate the $L^{r}$-norm of $f$ as 
 \begin{multline*}
  \norm{f}_{L^{r}(2^{n}R,2^{n_{*}+1}R)}\leq C\biggl(\int_{2^{n}R}^{2^{n_{*}+1}R}x^{-2r}\abs*{\int_{0}^{x}yf(y)\dy}^{r}\dx\biggr)^{1/r}\\*
  +\biggl(\int_{2^{n}R}^{2^{n_{*}+1}R}x^{-2r}\abs*{I[f](x)}^{r}\dx\biggr)^{1/r}.
 \end{multline*}
 Together with~\eqref{eq:decay:2} and Proposition~\ref{Prop:tail:averaged} we further obtain
 \begin{multline*}
   \norm{f}_{L^{r}(2^{n}R,2^{n_{*}+1}R)}\\*
   \leq C\biggl(\int_{2^{n}R}^{2^{n_{*}+1}R}x^{-r(1+\rho)}\dx\biggr)^{1/r}+C\biggl(\int_{2^{n}R}^{2^{n_{*}+1}R}x^{r(\lambda-2\rho-1)}\dx\biggr)^{1/r}\\*
   +C\biggl(\int_{2^{n}R}^{2^{n_{*}+1}R}x^{r(\alpha-1)}\abs[\bigg]{\int_{x/2}^{x}f(y)(x-y)^{-\rho}\dy}^{r}\dx\biggr)^{1/r}\\*
   +C\biggl(\int_{2^{n}R}^{2^{n_{*}+1}R}x^{r(\beta-1)}\abs[\bigg]{\int_{x/2}^{x}f(y)(x-y)^{-\rho}\dy}^{r}\dx\biggr)^{1/r}.
 \end{multline*}
 Computing the first two integrals explicitly and using Lemma~\ref{Lem:Young} we continue to estimate
 \begin{multline}\label{eq:decay:3}
     \norm{f}_{L^{r}(2^{n}R,2^{n_{*}+1}R)}\leq CR^{-1-\rho+\frac{1}{r}}+CR^{\lambda-2\rho-1+\frac{1}{r}}\\*
     +C\Bigl(R^{\alpha-1}+R^{\beta-1}\Bigr)\norm{f}_{L^{q}(2^{n-1}R,2^{n_{*}+1}R)}\norm{z^{-\rho}}_{L^{p}(0,2^{n_{*}+1}R)}.
 \end{multline}
Here $p,q$ and $r$ have to satisfy $1/p+1/q=1+1/r$. Next, we recall that $\rho>\lambda$ and $\beta\geq \alpha$ by assumption and thus we have 
\begin{equation}\label{eq:decay:4}
 R^{\lambda-2\rho-1+\frac{1}{r}}\leq R^{-1-\rho+\frac{1}{r}}\quad \text{and}\quad R^{\alpha-1}\leq R^{\beta-1} \qquad \text{for all } R\geq 1.
\end{equation}
Moreover, since $p\in(1,1/\rho)$ by~\eqref{eq:pos:choice:p} we note that it holds
 \begin{equation}\label{eq:decay:5}
  \norm{z^{-\rho}}_{L^{p}(0,2^{n_{*}+1}R)}\leq CR^{-\rho+\frac{1}{p}}.
 \end{equation}
 Thus, if we combine~\cref{eq:decay:3,eq:decay:4,eq:decay:5} it follows
  \begin{equation*}
  \norm{f}_{L^{r}(2^{n}R,2^{n_{*}+1}R)}\leq CR^{-1-\rho+\frac{1}{r}}+CR^{\beta-1-\rho+\frac{1}{p}}\norm{f}_{L^{q}(2^{n-1}R,2^{n_{*}+1}R)}.
 \end{equation*}
 Since $\beta<\rho$ by assumption, we can further estimate the second term on the right-hand side for $R\geq 1$ to get
 \begin{equation}\label{eq:decay:6}
  \norm{f}_{L^{r}(2^{n}R,2^{n_{*}+1}R)}\leq CR^{-1-\rho+\frac{1}{r}}+CR^{-1+\frac{1}{p}}\norm{f}_{L^{q}(2^{n-1}R,2^{n_{*}+1}R)}.
 \end{equation}
To conclude the proof for $0<\alpha\leq \beta$ we will now iteratively use this estimate. For this, we first take $n=0$, $r_{0}=p$ and $q_{0}=1$. Then as an immediate consequence of Proposition~\ref{Prop:tail:averaged} it follows
\begin{equation}\label{eq:est:L1:norm}
 \norm{f}_{L^{1}(R/2,2^{n_{*}+1}R)}=\int_{R/2}^{2^{n_{*}+1}R}\frac{1}{x}xf(x)\dx\leq CR^{-\rho}.
\end{equation}
Plugging this estimate in~\eqref{eq:decay:6} yields
\begin{equation*}
 \norm{f}_{L^{r_{0}}(R,2^{n_{*}+1}R)}\leq CR^{-1-\rho+\frac{1}{r_{0}}}\qquad \text{for all }R\geq 1.
\end{equation*}
In the next step we take $n=1$ and $q_{1}=r_{0}$ which requires to choose $r_{1}=\frac{p}{2-p}$. Thus, combining the previous estimate with~\eqref{eq:decay:6} it follows
\begin{equation*}
 \norm{f}_{L^{r_{1}}(2R,2^{n_{*}+1}R)}\leq C R^{-1-\rho+\frac{1}{r_{1}}}+CR^{-1+\frac{1}{p}+\frac{1}{r_{0}}-1-\rho}=CR^{-1-\rho+\frac{1}{r_{1}}}.
\end{equation*}
Proceeding in this way, we obtain for general $n<n_{*}$ that
\begin{equation*}
 \norm{f}_{L^{r_{n}}(2^{n}R,2^{n_{*}+1}R)}\leq CR^{-1-\rho+\frac{1}{r_{n}}}\qquad \text{with }r_{n}=\frac{p}{n(1-p)+1}.
\end{equation*}
 Thus, after a finite number of steps we get $r_{n_{*}}=\infty$ and it follows
 \begin{equation*}
  \norm{f}_{L^{\infty}(2^{n_{*}}R,2^{n_{*}+1}R)}\leq CR^{-1-\rho} \qquad \text{for all }R\geq 1.
 \end{equation*}
 This then proves the estimate~\eqref{eq:decay:0} for $0<\alpha\leq \beta$ if we choose $R_{*}=2^{n_{*}}$.
 
 To conclude the proof of the Proposition, it remains to show~\eqref{eq:I:decay:0} which is now an immediate consequence of~\eqref{eq:decay:0} and~\eqref{eq:decay:2}. Precisely, we have for $x\geq 2R_{*}$ that
 \begin{equation*}
  \begin{split}
   I[f](x)&\leq Cx^{1+\lambda-2\rho}+Cx^{\alpha-\rho}\int_{x/2}^{x}(x-y)^{-\rho}\dy+Cx^{\beta-\rho}\int_{x/2}^{x}(x-y)^{-\rho}\dy\\
   &=Cx^{1-\rho}\bigl(x^{\lambda-\rho}+x^{\alpha-\rho}+x^{\beta-\rho}\bigr).
  \end{split}
 \end{equation*}
 Thus, since $x\geq 1$ and $\lambda=\alpha+\beta$ with $0<\alpha\leq \beta$, we further deduce
 \begin{equation*}
  I[f](x)\leq Cx^{1-\rho-(\rho-\lambda)}
 \end{equation*}
 and the claim follows with $\delta=\rho-\lambda$ since $\lambda<\rho$ by assumption.
 \end{proof}

\subsection{The case $\alpha\leq 0\leq \beta$}

It remains to prove Proposition~\ref{Prop:pointwise:decay} for $\alpha\leq 0\leq \beta$ which is the most complicated case. The approach will be the same as in Section~\ref{Sec:pos:exp}, i.e.\@ we iteratively apply Lemma~\ref{Lem:Young} to get an estimate for the $L^{\infty}$-norm of $f$. However, depending on the values of $\rho,\lambda$ and $\beta$ the first bound that we obtain might not be enough and we have to start a second iteration to improve the $L^{\infty}$-bound further to conclude.

\begin{proof}[Proof of Proposition~\ref{Prop:pointwise:decay} for $\alpha\leq 0\leq\beta$]
 We first collect some estimates and fix some notation that we will need in the following, i.e.\@ we take $\nu\in(0,\min\{\rho-\lambda,1-\beta\})$. We then have the estimates
 \begin{align}
  \int_{x-y}^{\infty}z^{\alpha}f(z)\dz&\leq C\frac{(x-y)^{-\beta-\nu}}{(1+(x-y))^{\rho-\alpha-\beta-\nu}}\label{eq:mix:decay:1}\\
  \intertext{and}
  \int_{x-y}^{\infty}z^{\beta}f(z)\dz&\leq C\frac{(x-y)^{-\nu}}{(1+(x-y))^{\rho-\beta-\nu}}\label{eq:mix:decay:1.5}.
 \end{align}
To see this, we note that Lemma~\ref{Lem:moment:estimates} (i.e.\@ \cref{eq:moment:1,eq:moment:4}) yields
\begin{equation*}
 \begin{split}
  \int_{x-y}^{\infty}z^{\alpha}f(z)\dz&\leq (x-y)^{-\beta-\nu}\int_{x-y}^{\infty}z^{\alpha+\beta+\nu}f(z)\dz\\
  &\leq (x-y)^{-\beta-\nu}\begin{cases}
                           C_{\nu} & \text{if } (x-y)\leq 1\\
                           C_{\nu}(x-y)^{\alpha+\beta+\nu-\rho} & \text{if } (x-y)\geq 1.
                          \end{cases}
 \end{split}
\end{equation*}
From this, \eqref{eq:mix:decay:1} immediately follows, while for~\eqref{eq:mix:decay:1.5}, we can argue similarly using the splitting $z^{\beta}=z^{-\nu}z^{\beta+\nu}$. If we use now~\cref{eq:mix:decay:1,eq:mix:decay:1.5} together with the splitting in~\eqref{eq:decay:1}, we can deduce that 
\begin{multline}\label{eq:mix:decay:2}
 I[f](x)\leq C\int_{0}^{x/2}\bigl(y^{1+\alpha}x^{\beta-\rho}+y^{1+\beta}x^{\alpha-\rho}\bigr)f(y)\dy\\*
 +Cx^{1+\beta}\int_{x/2}^{x}f(y)\frac{(x-y)^{-\beta-\nu}}{(1+(x-y)^{\rho-\alpha-\beta-\nu}}\dy\\*
 +Cx^{1+\alpha}\int_{x/2}^{x}f(y)\frac{(x-y)^{-\nu}}{(1+(x-y))^{\rho-\beta-\nu}}\dy.
\end{multline}
Next, we use that from Lemma~\ref{Lem:moment:estimates} (i.e.\@ \eqref{eq:moment:2.5}) we have for all $x\geq 2$ that
\begin{equation*}
 \int_{0}^{x/2}y^{1+\alpha}f(y)\dy\leq Cx^{1-\rho}\qquad \text{and}\qquad \int_{0}^{x/2}y^{1+\beta}f(y)\dy\leq Cx^{1+\beta-\rho}
\end{equation*}
With these estimates and the fact that $\alpha\leq 0\leq \beta$ the first integral in~\eqref{eq:mix:decay:2} can be estimated for $x\geq 2$ as
\begin{equation*}
 \int_{0}^{x/2}\bigl(y^{1+\alpha}x^{\beta-\rho}+y^{1+\beta}x^{\alpha-\rho}\bigr)f(y)\dy\leq C\bigl(x^{1+\beta-2\rho}+x^{1+\alpha+\beta-2\rho}\bigr)\leq Cx^{1+\beta-2\rho}.
\end{equation*}
Thus, in summary we conclude from~\eqref{eq:mix:decay:2} for $x\geq 2$ that
\begin{multline}\label{eq:mix:I}
 I[f](x)\leq Cx^{1+\beta-2\rho}+Cx^{1+\beta}\int_{x/2}^{x}f(y)\frac{(x-y)^{-\beta-\nu}}{(1+(x-y)^{\rho-\alpha-\beta-\nu}}\dy\\*
 +Cx^{1+\alpha}\int_{x/2}^{x}f(y)\frac{(x-y)^{-\nu}}{(1+(x-y))^{\rho-\beta-\nu}}\dy.
\end{multline}
We then note that~\eqref{eq:self:sim} in terms of $I[f]$ can be written as
\begin{equation*}
 f(x)=(1-\rho)x^{-2}\int_{0}^{x}yf(y)\dy+x^{-2}I[f](x).
\end{equation*}
Thus, for $r\in[1,\infty)$ and for $n<n_{*}$ to be fixed later, we can compute 
\begin{multline*}
  \norm{f}_{L^{r}(2^{n}R,2^{n_{*}+1}R)}\\*
  \leq C\biggl(\int_{2^{n}R}^{2^{n_{*}+1}R}x^{-2r}\abs*{\int_{0}^{x}yf(y)\dy}^{r}\dx\biggr)^{1/r}+C\biggl(\int_{2^{n}R}^{2^{n_{*}+1}R}x^{r(\beta-1-2\rho)}\dx\biggr)^{1/r}\\*
  +C\biggl(\int_{2{n}R}^{2^{n_{*}+1}R}x^{r(\beta-1)}\abs[\bigg]{\int_{x/2}^{x}f(y)\frac{(x-y)^{-\beta-\nu}}{(1+(x-y))^{\rho-\alpha-\beta-\nu}}\dy}^{r}\dx\biggr)^{1/r}\\*
  +C\biggl(\int_{2^{n}R}^{2^{n_{*}+1}R}x^{r(\alpha-1)}\abs[\bigg]{\int_{x/2}^{x}f(y)\frac{(x-y)^{-\nu}}{(1+(x-y))^{\rho-\beta-\nu}}\dy}^{r}\dx\biggr)^{1/r}.
 \end{multline*}
Since $\int_{0}^{x}yf(y)\dy\leq Cx^{1-\rho}$ according to Proposition~\ref{Prop:tail:averaged} or~\ref{Prop:bounded:kernel} respectively, we obtain together with Lemma~\ref{Lem:Young} for $1/p+1/q=1/r+1$ that 
\begin{multline*}
  \norm{f}_{L^{r}(2^{n}R,2^{n_{*}+1}R)}\leq CR^{-1-\rho+\frac{1}{r}}+CR^{\beta-1-2\rho+\frac{1}{r}}\\*
  +CR^{\beta-1}\norm{f}_{L^{q}(2^{n-1}R,2^{n_{*}+1}R)}\norm*{\frac{z^{-\beta-\nu}}{(1+z)^{\rho-\alpha-\beta-\nu}}}_{L^{p}(0,2^{n_{*}+1}R)}\\*
  +CR^{\alpha-1}\norm{f}_{L^{q}(2^{n-1}R,2^{n_{*}+1}R)}\norm*{\frac{z^{-\nu}}{(1+z)^{\rho-\beta-\nu}}}_{L^{p}(0,2^{n_{*}+1}R)}.
 \end{multline*}
Since $\rho>\beta$ by assumption and we also assume $R\geq 1$, we have $R^{\beta-\rho}\leq 1$ such that we finally obtain
\begin{multline}\label{eq:mix:decay:3}
  \norm{f}_{L^{r}(2^{n}R,2^{n_{*}+1}R)}\\*
  \leq CR^{-1-\rho+\frac{1}{r}}+CR^{\beta-1}\norm{f}_{L^{q}(2^{n-1}R,2^{n_{*}+1}R)}\norm*{\frac{z^{-\beta-\nu}}{(1+z)^{\rho-\alpha-\beta-\nu}}}_{L^{p}(0,2^{n_{*}+1}R)}\\*
  +CR^{\alpha-1}\norm{f}_{L^{q}(2^{n-1}R,2^{n_{*}+1}R)}\norm*{\frac{z^{-\nu}}{(1+z)^{\rho-\beta-\nu}}}_{L^{p}(0,2^{n_{*}+1}R)}.
 \end{multline}
We now fix the parameter $p>1$ such that it holds
\begin{equation}\label{eq:choice:p}
 p=1+1/n_{*} \quad \text{with } n_{*}\in\N \text{ large enough such that}\quad p<\frac{1}{\beta+\nu}.
\end{equation}
Note that from the choice of $\nu$, we have $1/(\beta+\nu)>1$. Moreover, since $\beta\geq 0$ we have in particular that $p\nu<1$. With this, we find for $R\geq 1$ that
\begin{equation}\label{eq:mix:decay:4}
 \begin{split}
  \norm*{\frac{z^{-\beta-\nu}}{(1+z)^{\rho-\alpha-\beta-\nu}}}_{L^{p}(0,2^{n_{*}+1}R)}&\leq C\biggl(\int_{0}^{1}z^{-p(\beta+\nu)}\dz+\int_{1}^{2^{n_{*}+1}R}z^{p(\alpha-\rho)}\dz\biggr)^{1/p}\\
  &\leq C\bigl(1+R^{\alpha-\rho+\frac{1}{p}}\bigr)\leq CR^{\max\{\alpha-\rho+1/p,0\}}.
 \end{split}
\end{equation}
Similarly, we get
\begin{equation}\label{eq:mix:decay:5}
 \begin{split}
  \norm*{\frac{z^{-\nu}}{(1+z)^{\rho-\beta-\nu}}}_{L^{p}(0,2^{n_{*}+1}R)}
  \leq C R^{\max\{\beta-\rho+1/p,0\}}.
 \end{split}
\end{equation}
Combining~\cref{eq:mix:decay:3,eq:mix:decay:4,eq:mix:decay:5} and $R\geq 1$ we find
\begin{multline*}
  \norm{f}_{L^{r}(2^{n}R,2^{n_{*}+1}R)}\leq CR^{-1-\rho+\frac{1}{r}}\\*
  +C\bigl(R^{\max\{\lambda-\rho-1+1/p,\beta-1\}}+R^{\max\{\lambda-\rho-1+1/p,\alpha-1\}}\bigr)\norm{f}_{L^{q}(2^{n-1}R,2^{n_{*}+1}R)}.
 \end{multline*}
We now define $\chi\vcc=\max\{\lambda-\rho-1+1/p,\beta-1\}$ and note that it holds $\chi<0$ since $\beta<1$, $p>1$ and $\rho>\lambda$. Moreover, we get $R^{\max\{\lambda-\rho-1+1/p,\alpha-1\}}\leq R^{\chi}$ because $\alpha\leq \beta$ and $R\geq 1$. Thus, it further follows
\begin{equation}\label{eq:mix:decay:7}
 \norm{f}_{L^{r}(2^{n}R,2^{n_{*}+1}R)}\leq CR^{-1-\rho+\frac{1}{r}}+CR^{\chi}\norm{f}_{L^{q}(2^{n-1}R,2^{n_{*}+1}R)}.
\end{equation}
We now use this estimate iteratively until we we reach $r=\infty$. Thus, as a first step, we choose $n=0$, $q_{0}=1$ and $r_{0}=p$ and recall from~\eqref{eq:est:L1:norm} that \cref{Prop:tail:averaged,Prop:bounded:kernel} yield $\norm{f}_{L^{1}(R/2,2^{n_{*}+1}R)}\leq CR^{-\rho}$. Thus, we obtain from~\eqref{eq:mix:decay:7} that
\begin{equation}\label{eq:mix:decay:8}
 \norm{f}_{L^{r_{0}}(R,2^{n_{*}+1}R)}\leq CR^{-1-\rho+\frac{1}{r_{0}}}+CR^{\chi-\rho}.
\end{equation}
In the next step we take $n=1$ and $q_{1}=r_{0}$ which requires $r_{1}=p/(2-p)$. We then find together with~\cref{eq:mix:decay:7,eq:mix:decay:8} that
\begin{equation*}
 \begin{split}
  \norm{f}_{L^{r_{1}}(2R,2^{n_{*}+1}R)}&\leq CR^{-1-\rho+\frac{1}{r_{1}}}+CR^{\chi-1-\rho+\frac{1}{r_{0}}}+CR^{2\chi-\rho}\\
  &\leq CR^{-1-\rho+\frac{1}{r_{1}}}+CR^{2\chi-\rho}.
 \end{split}
\end{equation*}
In the last step we used that $\chi+1-1/p\leq 0$ which yields $\chi+1/r_{0}\leq 1/r_{1}$. If we iterate this procedure we obtain for $n<n_{*}$ and $q_{n}=r_{n-1}$ that 
\begin{equation*}
 \norm{f}_{L^{r_{n}}(2^{n}R,2^{n_{*}+1}R)}\leq CR^{-1-\rho+\frac{1}{r_{n}}}+CR^{(n+1)\chi-\rho}.
\end{equation*}
with $r_{n}=\frac{p}{1+n(1-p)}$. From the choice of $p$ in~\eqref{eq:choice:p} we have $n_{*}=1/(p-1)$ and thus $r_{n_{*}}=\infty$ which yields
\begin{equation}\label{eq:mix:decay:9}
 \norm{f}_{L^{\infty}(2^{n_{*}}R,2^{n_{*}+1}R)}\leq CR^{-1-\rho}+CR^{(n_{*}+1)\chi-\rho}.
\end{equation}
If $(n_{*}+1)\chi\leq -1$ this then already shows the estimate~\eqref{eq:decay:0} for $R_{*}=2^{n_{*}}$. 

On the other hand, if $(n_{*}+1)\chi> -1$ we have to iterate further. Precisely, we fix some constant $m\in\N$ such that 
\begin{equation*}
 (n_{*}+1)\chi+m\max\{\lambda-\rho,\beta-1\}\leq -1.
\end{equation*}
Then, we repeat the previous procedure with $\tilde{n}_{*}=n_{*}+m$ instead of $n_{*}$ which gives
\begin{equation}\label{eq:mix:decay:12}
 \norm{f}_{L^{\infty}(2^{n_{*}}R,2^{\tilde{n}_{*}+1}R)}\leq CR^{-1-\rho}+CR^{(n_{*}+1)\chi-\rho}.
\end{equation}
Moreover, we find that
\begin{equation}\label{eq:mix:decay:10}
 \begin{split}
  \norm[\bigg]{\frac{z^{-\beta-\nu}}{(1+z)^{\rho-\alpha-\beta-\nu}}}_{L^{1}(0,2^{\tilde{n}_{*}+1}R)}&\leq \biggl(\int_{0}^{1}z^{-\beta-\nu}\dz+\int_{1}^{2^{\tilde{n}_{*}+1}R}z^{\alpha-\rho}\dz\biggr)\\
  &\leq CR^{\max\{1+\alpha-\rho,0\}}.
 \end{split}
\end{equation}
Similarly, we get
\begin{equation}\label{eq:mix:decay:11}
 \begin{split}
  \norm[\bigg]{\frac{z^{-\nu}}{(1+z)^{\rho-\beta-\nu}}}_{L^{1}(0,2^{\tilde{n}_{*}+1}R)}\leq CR^{1+\beta-\rho}.
 \end{split}
\end{equation}
Thus, if we return to~\eqref{eq:mix:decay:3} and take $p=1$ and $r=q=\infty$ this yields together with~\cref{eq:mix:decay:10,eq:mix:decay:11} that
\begin{equation*}
 \begin{split}
  \norm{f}_{L^{\infty}(2^{n}R,2^{\tilde{n}_{*}+1}R)}&\leq CR^{-1-\rho}+C\bigl(R^{\max\{\lambda-\rho,\beta-1\}}+R^{\lambda-\rho}\bigr)\norm{f}_{L^{\infty}(2^{n-1}R,2^{\tilde{n}_{*}+1}R)}\\
  &\leq CR^{-1-\rho}+CR^{\max\{\lambda-\rho,\beta-1\}}\norm{f}_{L^{\infty}(2^{n-1}R,2^{\tilde{n}_{*}+1}R)}.
 \end{split}
\end{equation*}
If we now iterate this estimate $m$ times and use~\eqref{eq:mix:decay:12}, it follows
\begin{equation*}
 \begin{split}
  \norm{f}_{L^{\infty}(2^{n_{*}+m}R,2^{n_{*}+m+1}R)}&\leq CR^{-1-\rho}+CR^{m\max\{\lambda-\rho,\beta-1\}+(n_{*}+1)\chi-\rho}\leq CR^{-1-\rho}.
 \end{split}
\end{equation*}
The last step follows from the choice of $m$ together with the assumption $R\geq 1$. The estimate~\eqref{eq:decay:0} then directly follows with $R_{*}=2^{n_{*}+m}$. 

The upper bound~\eqref{eq:I:decay:0} on $I[f]$ is then an immediate consequence. Precisely, since the estimate~\eqref{eq:decay:0} directly implies $f(x)\leq Cx^{-1-\rho}$ for $x\geq R_{*}$ we deduce from \eqref{eq:mix:I} in combination with~\cref{eq:mix:decay:10,eq:mix:decay:11} that 
\begin{multline*}
  I[f](x)\leq Cx^{1+\beta-2\rho}+C\bigl(x^{1+\beta+\max\{1+\alpha-\rho,0\}}+x^{1+\alpha+1+\beta-\rho}\bigr)x^{-1-\rho}\\*
  =C\bigl(x^{\beta-\rho}+x^{\max\{\lambda-\rho,\beta-1\}}+x^{\lambda-\rho}\bigr)x^{1-\rho}.
 \end{multline*}
Thus, the estimate~\eqref{eq:I:decay:0} holds with $\delta=\rho-\beta$ since $\lambda=\alpha+\beta\leq \beta<\rho<1$.
\end{proof}

\subsection{Proof of Theorem~\ref{Thm:asymptotics}}

The proof of Theorem~\ref{Thm:asymptotics} follows now easily from Proposition~\ref{Prop:pointwise:decay}.

\begin{proof}[Proof of Theorem~\ref{Thm:asymptotics}]
 To show the continuity of self-similar profiles, one can proceed in the same way as in Lemma~4.2 of~\cite{NiV13a}, i.e.\@ one first proves that $f\in L^{\infty}_{\mathrm{loc}}((0,\infty))$ using similar estimates as in the proof of Proposition~\ref{Prop:pointwise:decay}. Once this is established, it is straightforward to show that $f$ is even continuous on $(0,\infty)$. Since the corresponding estimates are quite similar, we do not give further details on this here but only refer to~\cite{NiV13a} (see also Lemma~3.3 in~\cite{NTV16a}).
 
 Thus, it only remains to show how the asymptotic behaviour of self-similar profiles can be obtained from~\cref{Prop:asymptotics:averaged,Prop:pointwise:decay} (see also Proposition~3.6 in~\cite{NTV16a} and Lemma~4.3 in~\cite{NiV13a}).
 
 We recall that in terms of $I[f]$ the equation~\eqref{eq:self:sim} reads
 \begin{equation*}
  x^{2}f(x)=(1-\rho)\int_{0}^{x}yf(y)\dy+I[f](x).
 \end{equation*}
 This can be rewritten as
 \begin{equation}\label{eq:asymptotics}
  \frac{x^{1+\rho}f(x)}{1-\rho}=x^{\rho-1}\int_{0}^{x}yf(y)\dy+\frac{1}{1-\rho}x^{\rho-1}I[f](x).
 \end{equation}
 Due to Proposition~\ref{Prop:pointwise:decay} we have that $x^{\rho-1}I[f](x)\to 0$ as $x\to \infty$ while due to the rescaling in Proposition~\ref{Prop:asymptotics:averaged} it holds $x^{\rho-1}\int_{0}^{x}yf(y)\dy\to 1$ for $x\to\infty$. Thus, taking the limit $x\to \infty$ in~\eqref{eq:asymptotics} it follows
 \begin{equation*}
  \frac{x^{1+\rho}f(x)}{1-\rho}\longrightarrow 1\qquad \text{for }x\longrightarrow \infty.
 \end{equation*}
 This then finishes the proof.
\end{proof}

\section*{Acknowledgements} 
This work has been supported through the CRC 1060 \textit{The mathematics of emergent effects} at the University of Bonn that is funded through the German Science Foundation (DFG). Moreover, partial support through a Lichtenberg Professorship Grant of the VokswagenStiftung awarded to Christian Kühn is acknowledged.

\end{document}